\documentclass{amsart}
\usepackage{mathrsfs}

\newtheorem{Thm}{Theorem}
\newtheorem{Prop}[Thm]{Proposition}
\newtheorem{Cor}[Thm]{Corollary}
\newtheorem{Lemma}[Thm]{Lemma}
\newtheorem*{Question}{Question}

\theoremstyle{remark}
\newtheorem{Remark}[Thm]{Remark}

\theoremstyle{definition}
\newtheorem*{Def}{Definition}

\title[Degree-inverting involution on $M_n(\mathbb{F})$ and $UT_n(\mathbb{F})$]{Degree-inverting involution on full square and triangular matrices}
\author{Lais S. da Fonseca}
\address{Department of Mathematics, State University of Maring\'a, Maring\'a, PR, Brazil.}
\email[L. Spada]{laisspada2@gmail.com}
\author{Ednei A. Santulo Jr.}
\address{Department of Mathematics, State University of Maring\'a, Maring\'a, PR, Brazil.}
\email[E. A. Santulo Jr.]{easjunior@uem.br}
\author{Felipe Y. Yasumura}
\address{Department of Mathematics, State University of Maring\'a, Maring\'a, PR, Brazil.}
\email[F. Yasumura]{felipeyukihide@gmail.com}
\thanks{L. S. da Fonseca was financed by the Coordena\c c\~ao de Aperfei\c coamento de Pessoal de N\'ivel Superior - Brasil (CAPES) - Finance Code 001}
\thanks{F. Y. Yasumura was financed by the Coordena\c c\~ao de Aperfei\c coamento de Pessoal de N\'ivel Superior - Brasil (CAPES) - Finance Code 001}
\begin{document}
\begin{abstract}
In this short note, we classify the degree-inverting involution on the full square and triangular matrices.
\end{abstract}
\maketitle

\section{Introduction}
Graded rings appear naturally in several branches of Mathematics and Physics. For instance, one can construct a grading from a filtered algebra, a local valuation ring, a polynomial ring, an automorphism of finite order of an algebra, a finite-dimensional Lie algebra over an algebraically closed field of characteristic zero, etc.

Recall that a $G$-grading on an algebra $\mathcal{A}$ is a vector space decomposition $\mathcal{A}=\bigoplus_{g\in G}\mathcal{A}_g$ such that $\mathcal{A}_g\mathcal{A}_h\subseteq\mathcal{A}_{gh}$, for all $g,h\in G$. Some graded algebras are endowed also with a graded involution, in the following sense: $\psi$ is an involution of $\mathcal{A}$, and $\psi(\mathcal{A}_g)\subseteq\mathcal{A}_{g^{-1}}$, for any $g\in G$. Examples include: the usual transposition of square matrices with any good grading, the usual involution on Leavitt path algebras endowed with their usual grading, etc. It is worth mentioning that, in a recent work \cite{HazV}, the authors proved that the graded involution enriches the structure of the Graded Grothendieck group of a graded ring. Thus, understanding gradings and graded involution on a given algebra seems to be an interesting problem. From now on, we will refer the graded involution as \emph{degree-inverting involution}.

On the other hand, in \cite{BSZ2005}, the authors call a \emph{graded involution} an involution satisfying $\psi(\mathcal{A}_g)\subseteq\mathcal{A}_g$, for all $g\in G$. They proved that the degree-preserving involutions are fundamental to give a description of group gradings on some simple Lie algebras, a question raised by Patera and Zassenhaus \cite{PZ1989}. After the contribution of several authors, the classification of the degree-presernving involutions on matrix algebras, and the gradings on simple Lie algebras are essentially complete \cite{BK2010,BSZ2001,BSZ2005,BZ2002,DM2006} (among others), see also the monograph \cite{EK2013}. Thus, degree-preserving involution is an essential tool as well.

In this paper, using the ideas of the degree-preserving case \cite{Elduque,BK2010} (see also \cite{EK2013}), we classify degree-inverting involutions on matrix algebras and on upper triangular matrices, improving the results of \cite{FM2017}.

This paper is divided as follows: we include a few preliminary theory in Section \ref{preliminaries}. Then, we provide partial results for degree-inverting involution on graded division algebras (Section \ref{grad_div_alg}). Next, we copy the theory presented in \cite[Section 2.4]{EK2013} (see also the paper by Elduque \cite{Elduque}) to study the matrix algebra case in Section \ref{matrix}. Finally, in Section \ref{utn}, we obtain results for the upper triangular matrices case.

\section{Preliminaries\label{preliminaries}}
\subsection{Graded Algebras}
We shall work with graded algebras rather than graded rings, as follows. Let $G$ be any group. We say that an algebra $\mathcal{A}$ is $G$-graded if there exists a vector-space decomposition $\mathcal{A}=\bigoplus_{g\in G}\mathcal{A}_g$ such that $\mathcal{A}_g\mathcal{A}_h\subseteq\mathcal{A}_{gh}$, for all $g,h\in G$. The subspace $\mathcal{A}_g$ is called \emph{homogeneous component of degree $g$}. A nonzero element $x\in\mathcal{A}_g$ is called a homogeneous element of degree $g$. We denote $\deg x=g$.

A map $f:\mathcal{A}\to\mathcal{B}$ between two $G$-graded algebras is called a \emph{graded homomorphism} if $f$ is an algebra homomorphism, and $f(\mathcal{A}_g)\subseteq\mathcal{B}_g$ for all $g\in G$. If, moreover, $f$ is an isomorphism, then $f$ is called a graded isomorphism; in this case, $\mathcal{A}$ and $\mathcal{B}$ are said to be isomorphic.

A \emph{graded division algebra} is an associative algebra $\mathcal{D}$ with $1$, where each nonzero homogeneous element $x\in\mathcal{D}$ is invertible. 

Now let $\mathcal{R}=M_n(\mathbb{F})$ be a matrix algebra endowed with a $G$-grading. Then the graded version of the Density Theorem tells us that we can find a graded division algebra $\mathcal{D}$, $\dim\mathcal{D}=\ell^2$, and a sequence $(g_1,\ldots,g_m)$, such that $\mathcal{R}\cong M_m(\mathbb{F})\otimes\mathcal{D}$, where the grading is given by
\begin{equation}\label{grad}
\deg e_{ij}\otimes d=g_i\deg(d)g_j^{-1},\quad d\in\mathcal{D}\text{ homogeneous}.
\end{equation}

Let $\mathcal{A}=\bigoplus_{g\in G}\mathcal{A}_g$ be a $G$-graded algebra. We say that $V$ is a \emph{$G$-graded right $\mathcal{A}$-module} if $V$ is a right $\mathcal{A}$-module and there exists a vector space decomposition $V=\bigoplus_{g\in G}V_g$ such that $V_h\mathcal{A}_g\subseteq V_{hg}$, for all $g,h\in G$. Similarly we define the notion of graded left modules.

Given two $G$-graded right $\mathcal{A}$-modules $V$ and $W$, we say that $f:V\to W$ is a \emph{graded map of degree $g$} if $f$ is an $\mathcal{A}$-homomorphism, and $f(V_h)\subseteq W_{gh}$, for all $h\in G$. The graded maps of degree $1$ are also known as graded $\mathcal{A}$-homomorphism. We denote $\mathrm{Hom}_g(V,W)$ the set of all graded maps of degree $g$; and let $\mathrm{Hom}_\text{gr}(V,W)=\bigoplus_{g\in G}\mathrm{Hom}_g(V,W)$. If $V$ and $W$ are finite-dimensional, then we have $\mathrm{Hom}(V,W)=\mathrm{Hom}_\text{gr}(V,W)$, so $\mathrm{Hom}(V,W)$ gets a vector-space grading.

Now, let $\mathcal{D}$ be a finite-dimensional graded division algebra, and let $V$ be a finite-dimensional $G$-graded right $\mathcal{D}$-module. Then $\mathcal{R}=\mathrm{End}_\mathcal{D}(V)=\mathrm{Hom}(V,V)$ is a $G$-graded algebra isomorphic to a matrix algebra. Moreover, $V$ is a $G$-graded left $\mathcal{R}$-module.

Finally, we provide a precise definition of the following:
\begin{Def}
Let $\mathcal{A}=\bigoplus_{g\in G}\mathcal{A}_g$ be a $G$-graded algebra. An involution $\psi$ on $\mathcal{A}$ is a \emph{degree-inverting involution} if $\psi(\mathcal{A}_g)\subseteq\mathcal{A}_{g^{-1}}$, for all $g\in G$.
\end{Def}
In this paper, \emph{involution} will mean a first kind involution, that is, a $C(\mathcal{A})$-linear map, where $C(\mathcal{A})$ is the center of $\mathcal{A}$.

\subsection{Factor sets}
Let $T$ be a finite group. A map $\sigma:T\times T\to\mathbb{F}^\times$, where $\mathbb{F}^\times$ is the set of invertible elements of $\mathbb{F}$, is called a \emph{2-cocycle} or a \emph{factor set} if
$$
\sigma(u,v)\sigma(uv,w)=\sigma(u,vw)\sigma(v,w),\quad\forall u,v,w\in T.
$$
These objects are interesting and intensively studied in the context of cohomology of groups (see, for instance, \cite{Kap1,Kap2}). However, we do not need such generalities, and we limit ourselves within the theory we will need.

Denote by $Z^2(T,\mathbb{F}^\times)$ the set of all factor sets. Since $\mathbb{F}^\times$ is commutative with respect to the product, the $Z^2(T,\mathbb{F}^\times)$ acquires a natural structure of abelian group, by point-wise product.

We can construct algebras from factor sets. Given an arbitrary map $\sigma:T\times T\to\mathbb{F}^\times$ denote by $\mathbb{F}^\sigma T$ the following algebra: $\mathbb{F}^\sigma T$ has a basis $\{X_u\mid u\in T\}$, and the product is defined by $X_uX_v=\sigma(u,v)X_{uv}$. Note that $\mathbb{F}^\sigma T$ is associative if and only if $\sigma\in Z^2(T,\mathbb{F}^\times)$. For instance, if $\sigma=1$ (the constant function), then $\mathbb{F}^\sigma T$ is the group algebra of $T$. Next, we investigate the isomorphism classes of algebras given by factor sets.

For any arbitrary map $\lambda:T\to\mathbb{F}^\times$, we obtain a factor set $\delta\lambda$ by the formula
$$
\delta\lambda(uv):=\frac{\lambda(u)\lambda(v)}{\lambda(uv)}.
$$
Since $\delta(\lambda_1\lambda_2)=\delta\lambda_1\delta\lambda_2$, $B^2(T,\mathbb{F}^\times):=\{\delta\lambda\mid\lambda:T\to\mathbb{F}^\times\}$ is a subgroup of $Z^2(T,\mathbb{F}^\times)$. We denote the quotient by $H^2(T,\mathbb{F}^\times)=Z^2(T,\mathbb{F}^\times)/B^2(T,\mathbb{F}^\times)$, and call it the \emph{second cohomology group of $T$}. Given $\sigma\in Z^2(T,\mathbb{F}^\times)$, we denote by $[\sigma]$ the element $\sigma B^2(T,\mathbb{F}^\times)$ in $H^2(T,\mathbb{F}^\times)$.
\begin{Lemma}[{\cite[Chapter 2, Lemma 1.1]{Kap1}}]
Let $\sigma_1,\sigma_2\in Z^2(T,\mathbb{F}^\times)$. Then $\mathbb{F}^{\sigma_1}T\cong\mathbb{F}^{\sigma_2}T$ if and only if $[\sigma_1]=[\sigma_2]$.\qed
\end{Lemma}

The following is an easy manipulation:
\begin{Lemma}
Let $[\sigma]\in H^2(T,\mathbb{F}^\times)$. Then, there exists $\sigma'\in[\sigma]$ such that $\sigma'(u,1)=\sigma'(1,u)=1$, for all $u\in T$.\qed
\end{Lemma}
Hence, combining the two previous result, given $\mathbb{F}^\sigma T$, we can assume that $\sigma(u,1)=\sigma(1,u)=1$, for all $u\in T$.

Finally, it is worth mentioning that, if $\mathrm{char}\,\mathbb{F}$ does not divide $|T|$, then $\mathbb{F}^\sigma T$ is semiprimitive (that is, its Jacobson radical is zero).

\subsection{Graded division algebras}
Graded division algebras have a nice description when the base field is algebraically closed. Assume that $\mathbb{F}$ is algebraically closed and let $\mathcal{D}=\bigoplus_{g\in G}\mathcal{D}_g$ be a finite-dimensional graded division algebra over $\mathbb{F}$. Let $T=\{g\in G\mid\mathcal{D}_g\ne0\}$ be its support. Then it is easy to see that $T$ is a subgroup of $G$. We use multiplicative notation for the product of $T$, and denote by $1$ its neutral element.

Moreover, $\mathcal{D}_1\supseteq\mathbb{F}$ is a division algebra. So $\mathcal{D}_1=\mathbb{F}$, since $\mathbb{F}$ is algebraically closed and $\dim_\mathbb{F}\mathcal{D}_1<\infty$. This also implies $\dim\mathcal{D}_g=1$, for all $g\in T$. Let $\{X_u\mid u\in T\}$ be a homogeneous basis of $\mathcal{D}$. Then $X_uX_v=\sigma(u,v)X_{uv}$, for some $\sigma(u,v)\in\mathbb{F}^\times$. Since $\mathcal{D}$ is associative, from $(X_uX_v)X_w=X_u(X_vX_w)$, we derive that $\sigma$ is a 2-cocycle. Hence, $\mathcal{D}\cong\mathbb{F}^\sigma T$, the twisted group algebra of $T$ by $\sigma$. Conversely, for any finite group $T$ and any $\sigma\in Z^2(T,\mathbb{F}^\times)$, the natural $T$-grading on $\mathbb{F}^\sigma T$ turns it into a graded division algebra.

Now, assume that $T$ is abelian. Let $\beta(u,v)=\sigma(u,v)\sigma(v,u)^{-1}$. A direct computation shows that $\beta$ is an alternating bicharacter; moreover, $\mathcal{D}$ is central if and only if $\beta$ is nondegenerate. Finally, Theorem 2.15 of \cite{EK2013} tells that the pair $(T,\beta)$ uniquely determines an isomorphism class of finite-dimensional central graded division algebras over $\mathbb{F}$ with commutative support. Hence, if $T$ is abelian, the pairs $(T,\beta)$ are in bijection with the elements of the second cohomology group $H^2(T,\mathbb{F}^\times)$.

\subsection{Realization of graded division algebras with commutative support\label{abelian_graded_division_algebra}}
Let $\varepsilon$ be a primitive $n$-root of unity. Consider the elements
\begin{equation}\label{epsilongrad}
X=\left(\begin{array}{cccccc}
0&1&0&\cdots&0&0\\%
0&0&1&\cdots&0&0\\%
\vdots&\vdots&\vdots& &\vdots&\vdots\\%
0&0&0&\cdots&0&1\\%
1&0&0&\cdots&0&0\end{array}\right),\quad
Y=\left(\begin{array}{ccccc}%
\varepsilon^{n-1}& & & &0\\%
&\varepsilon^{n-2}\\%
&&\ddots\\%
&&&\varepsilon\\%
0&&&&1%
\end{array}\right).
\end{equation}
Note that $\varepsilon XY=YX$ and $X^n=Y^n=1$. Moreover, $\{X^iY^j\mid i,j=0,1,\ldots,n\}$ is a vector space basis of $M_n(\mathbb{F})$. Also, $\mathcal{A}_{(i,j)}=\mathrm{Span}\{X^iY^j\}$ constitute a $\mathbb{Z}_n\times\mathbb{Z}_n$-grading on $M_n(\mathbb{F})$. This grading is called $\varepsilon$-grading, and it is a division grading.

Now, if $M_n(\mathbb{F})$ is endowed with a division grading, then, as mentioned in the previous section, the support $T$ of the grading is a group, and the product is determined by a non-degenerate alternating bicharacter $\beta:T\times T\to\mathbb{F}^\times$. Thus, we obtain a decomposition $T=H_1^2\times H_2^2\times\cdots\times H_s^2$, where each $H_i$ is $\beta$-invariant and $H_i\cong\mathbb{Z}_{m_i}$. Moreover, we obtain
\begin{equation}\label{kronecker}
M_n=M_{m_1}\otimes M_{m_2}\otimes\cdots\otimes M_{m_s},
\end{equation}
where $\mathrm{Supp}\,M_{m_i}=H_i^2$ and $M_{m_i}$ has an $\varepsilon_i$-grading (see \cite[Section 2.2]{EK2013} for more details).

Thus, if $\mathcal{D}$ is a central finite-dimensional graded division algebra over an algebraically closed field $\mathbb{F}$, then we can realize $\mathcal{D}$ as a matrix algebra. Such realization is made after a choice of Kronecker product identification as in \eqref{kronecker}, and, for each $M_{m_i}$, a choice of a basis as in \eqref{epsilongrad}.

\section{Degree-inverting involution on graded division algebras\label{grad_div_alg}}
As mentioned above, over an algebraically closed field $\mathbb{F}$, a finite-dimensional $G$-graded-division algebra assumes the form $\mathbb{F}^\sigma T$, where $T\subseteq G$ is a finite subgroup, and $\sigma:T\times T\to\mathbb{F}^\times$ is a 2-cocycle.

\begin{Lemma}\label{llem1}
Given $\sigma\in Z^2(T,\mathbb{F}^\times)$, let $\bar\sigma:T\times T\to\mathbb{F}^\times$ be defined by $\bar{\sigma}(u,v)=\sigma(v^{-1},u^{-1})$. Then $[\bar\sigma]=[\sigma^{-1}]$.
\end{Lemma}
\begin{proof}
We have
\begin{align*}
\sigma(u,v)\bar{\sigma}(u,v)&=\sigma(u,v)\sigma(v^{-1},u^{-1})\\%
&=\sigma(uv,v^{-1})^{-1}\sigma(v,v^{-1})\sigma(u,vv^{-1})\sigma(v^{-1},u^{-1}).
\end{align*}
Also,
\begin{align*}
\sigma(uv,v^{-1})=\sigma(uvv^{-1},u^{-1})^{-1}\sigma(v^{-1},u^{-1})\sigma(uv,v^{-1}u^{-1}).
\end{align*}
Thus, continuing from the first equation,
\begin{align*}
\sigma(u,v)\bar{\sigma(u,v)}&=\sigma(u,u^{-1})\sigma(v,v^{-1})\sigma(uv,(uv)^{-1})^{-1}\\%
&=\delta\lambda(u,v),
\end{align*}
where $\lambda(u):=\sigma(u,u^{-1})$.
\end{proof}

We fix a $\sigma\in Z^2(T,\mathbb{F}^\times)$, and a homogeneous basis $\{X_u\mid u\in T\}$ of $\mathbb{F}^\sigma T$.
\begin{Prop}
$\mathbb{F}^\sigma T$ admits a degree-inverting involution if and only if $[\sigma]^2=1$.
\end{Prop}
\begin{proof}
Assume that $\rho$ is a degree-inverting involution on $\mathbb{F}^\sigma T$. Let $\mu:T\to\mathbb{F}^\times$ be such that $\rho(X_u)=\mu(u)X_{u^{-1}}$, for all $u\in T$. Note that, for any $u,v\in T$,
\begin{align*}
&\rho(X_uX_v)=\rho(X_v)\rho(X_u)=\mu(u)\mu(v)\sigma(v^{-1},u^{-1})X_{v^{-1}u^{-1}},\\%
&\rho(X_uX_v)=\sigma(u,v)\rho(X_{uv})=\sigma(u,v)\mu(uv)X_{(uv)^{-1}}.
\end{align*}
Thus $\sigma=(\delta\mu)\bar{\sigma}$, which implies $[\sigma]^2=1$, by Lemma \ref{llem1}.

Conversely, if $[\sigma]=[\sigma^{-1}]=[\bar{\sigma}]$, let $\mu:T\to\mathbb{F}^\times$ be such that $\sigma=(\delta\mu)\bar{\sigma}$. We claim that $\rho:\mathbb{F}^\sigma T\to\mathbb{F}^\sigma T$ defined by $\rho(X_u)=\mu(u)X_{u^{-1}}$ is a degree-inverting involution. By definition, $\rho$ inverts the degrees, so we only need to show that it is an involution. We have
\begin{align*}
&\rho(X_uX_v)=\sigma(u,v)\mu(uv)X_{(uv)^{-1}},\\%
&\rho(X_v)\rho(X_u)=\mu(v)\mu(u)\sigma(v^{-1},u^{-1})X_{v^{-1}u^{-1}},
\end{align*}
and both coincide by the choice of $\mu$. Finally,
$$
\rho\rho(X_u)=\mu(u)\mu(u^{-1})X_u.
$$
So, we need to show that $\mu(u)\mu(u^{-1})=1$, for all $u\in T$. However, we note that, for any $u,v\in T$, we have
$$
\frac{\mu(u)\mu(v)}{\mu(uv)}=\frac{\sigma(u,v)}{\bar{\sigma}(u,v)}.
$$
In particular, $\mu(u)\mu(u^{-1})=\sigma(u,u^{-1})\bar{\sigma}(u,u^{-1})^{-1}\mu(uu^{-1})=\mu(1)$, for any $u\in T$. Taking $u=1$, we obtain $\mu(1)=1$. Hence, $\mu(u)\mu(u^{-1})=1$, for any $u\in T$, and we are done.
\end{proof}

\begin{Lemma}
There exists an isomorphism $\mathrm{Aut}_G(\mathbb{F}^\sigma T)\cong\mathrm{Hom}(T,\mathbb{F}^\times)$.
\end{Lemma}
\begin{proof}
Given $\psi\in\mathrm{Aut}_G(\mathbb{F}^\sigma T)$, we have $\psi(X_u)=\chi(u)X_u$, for some $\chi:T\to\mathbb{F}^\times$, for all $u\in T$. It is easy to check that $\chi$ is a group homomorphism. Conversely, given $\chi:T\to\mathbb{F}^\times$, the map $\psi$ defined by $\psi(X_u)=\chi(u)X_u$ is a $G$-graded automorphism of $\mathbb{F}^\sigma T$. So, we obtain a bijection $\psi\mapsto\chi$.

Finally, note that, if $\psi_i\mapsto\chi_i$, for $i=1,2$, then $\psi_1\psi_2\mapsto\chi_1\chi_2$. So, the bijection is a group isomorphism.
\end{proof}
Denote $\hat T=\mathrm{Hom}(T,\mathbb{F}^\times)$. As a consequence of the previous lemma, $\mathrm{Aut}_G(\mathbb{F}^\sigma T)\cong\hat T$ is an abelian group.

\begin{Lemma}
Let $\rho$ be a degree-inverting involution on $\mathbb{F}^\sigma T$. Then, for any $\psi\in\mathrm{Aut}_G(\mathbb{F}^\sigma T)$, $\rho\circ\psi$ is a degree-inverting involution on $\mathbb{F}^\sigma T$. Every degree-inverting involution is obtained by such way.
\end{Lemma}
\begin{proof}
Using that $\psi(X_u)=\chi(u)X_u$, for all $u\in T$, we obtain that $\rho\circ\psi$ is a degree-inverting involution by direct computation. If $\rho'$ is another degree-inverting involution, then $\rho\rho'$ is a graded automorphism, thus $\rho\rho'=\psi$, for some $\psi\in\mathrm{Aut}_G(\mathbb{F}^\sigma T)$. Thus, $\rho'=\rho\circ\psi$.
\end{proof}

Given a group $H$, we denote $S(H)=\{h^2\mid h\in H\}$. Notice that, if $H$ is abelian, then $S(H)$ is a subgroup of $H$.
\begin{Lemma}
$\rho$ and $\rho\circ\psi$ are equivalent if and only if $\psi\in S(\mathrm{Aut}_G(\mathbb{F}^\sigma T))$.
\end{Lemma}
\begin{proof}
For any $\psi\in\mathrm{Aut}_G(\mathbb{F}^\sigma T)$, note that $\rho\circ\psi=\psi^{-1}\circ\rho$.

So, if $\psi=\varphi^2$, for some $\varphi\in\mathrm{Aut}_G(\mathbb{F}^\sigma T)$, then
$$
\rho\psi=\rho\varphi\varphi=\varphi^{-1}\rho\varphi,
$$
which shows that $\rho\psi\sim\rho$. Conversely, assume that $\rho\psi=\varphi^{-1}\rho\varphi$, for some $\varphi$. Then we obtain $\rho\psi=\rho\varphi^2$, which implies $\psi=\varphi^2\in S(\mathrm{Aut}_G(\mathbb{F}^\sigma T))$.
\end{proof}

We summarize the results.
\begin{Thm}
Let $\mathbb{F}$ be a field, $T$ a finite group, and $\sigma:T\times T\to\mathbb{F}^\times$ a 2-cocycle. Then $\mathbb{F}^\sigma T$ admits a degree-inverting involution if and only if $[\sigma]^2=1$. In this case, there exist $|\hat{T}/S(\hat{T})|$ non-equivalent classes of degree-inverting involution on $\mathbb{F}^\sigma T$.\qed
\end{Thm}

Now, we are interested in the case where we have simultaneously $\mathbb{F}^\sigma T$ isomorphic to a matrix algebra, and $[\sigma]$ of order 2. The last one can be achieved if we compute the Schur multiplier $M(T)$. The former one is equivalent to: (a) $|T|=n^2$, for some $n$, and (b) $T$ admits an irreducible (projective) $\sigma$-representation of degree $n$.

Although some works were dedicated to either answer the first question, or to compute the Schur multiplier (see, for instance, \cite{Kap1,Kap2}), we were not able to find a single example of a non-abelian group satisfying both conditions. So we leave the following question.

\begin{Question}
Find a non-abelian finite group $T$ of order $n^2$, for some $n\in\mathbb{N}$, and a 2-cocyle $\sigma:T\times T\to\mathbb{F}^\times$ such that $[\sigma]^2=1$, and $\mathbb{F}^\sigma T\cong M_n(\mathbb{F})$.
\end{Question}

\subsection{Abelian case}
Things become easier if we assume a priori the grading group abelian.

The following was essentially proved in \cite{FM2017}:
\begin{Lemma}\label{lem9}
Let $\psi_0:\mathcal{D}\to\mathcal{D}$ be a degree-inverting anti-automorphism, where $\mathcal{D}$ is a central finite-dimensional graded division algebra with support $T$, where $T$ is an abelian group. Then $T$ is an elementary $2$-group.
\end{Lemma}
\begin{proof}
As mentioned in Subsection \ref{abelian_graded_division_algebra}, $\mathrm{Supp}\,\mathcal{D}=H_1^2\times H_2^2\times\cdots\times H_s^2$, where each $H_i\cong\mathbb{Z}_{n_i}$, and $\mathcal{D}\cong M_{n_1}\otimes M_{n_2}\otimes\cdots\otimes M_{n_s}$, where each $M_{n_i}$ is endowed with an $\varepsilon_i$-grading.

Since every nonzero homogeneous component of $\mathcal{D}$ has dimension $1$, we see that each $1\otimes\cdots1\otimes M_{n_i}\otimes1\cdots\otimes1$ is invariant under the anti-automorphism, with support $1\times\cdots1\times H_i^2\times1\cdots\times1$. From Lemma 4.6 of \cite{FM2017}, we obtain $n_i=2$ and $H_i\cong\mathbb{Z}_2$.
\end{proof}
So, an immediate consequence is the following remark:
\begin{Cor}\label{commutativegradeddiv}
Let $\mathcal{D}$ be a central finite-dimensional graded division algebra over an algebraically closed field $\mathbb{F}$, and assume that $\mathrm{Supp}\,\mathcal{D}$ is commutative. Then an involution on $\mathcal{D}$ is a degree-preserving involution if and only if it is a degree-inverting involution.\qed
\end{Cor}

\section{Degree-inverting involution on matrix algebras\label{matrix}}
In this section we investigate degree-inverting involution on matrix algebras over an algebraically closed field. The arguments in this section are a copy of the ordinary case \cite[Section 2.4]{EK2013} (see also the original paper by Elduque \cite{Elduque}). If a matrix algebra is endowed with a grading and a degree-inverting involution, then its support does not need to be commutative. This is a contrast with the degree-preserving involution case (see, for instance, \cite[Proposition 2.49]{EK2013}).

We fix an algebraically closed field $\mathbb{F}$ and an arbitrary group $G$. Let $\mathcal{D}$ be a finite-dimensional $G$-graded division algebra, and let $T$ be its support (then $T\subseteq G$ is a finite subgroup). Let $V$ be a finite-dimensional $G$-graded right $\mathcal{D}$-module. We define
$$
V^\ast=\{f:V\to\mathcal{D},\text{ $f$ is a graded $\mathcal{D}$-linear map}\}.
$$
Thus, $V^\ast$ has a natural $G$-grading. For homogeneous $f\in V^\ast$ and $v\in V$, we denote $\langle f,v\rangle=f(v)$ to emphasize the duality between $V$ and $V^\ast$. Moreover, one has
$$
\deg\langle f,v\rangle=\deg f\deg v.
$$
Let $\mathcal{R}=\mathrm{End}_\mathcal{D}(V)$. Then $\mathcal{R}$ is a matrix algebra endowed with a $G$-grading. The natural action of $\mathcal{R}$ on $V$ turns $V$ a graded left $\mathcal{R}$-module. Also, $V^\ast$ has a structure of graded right $\mathcal{R}$-module given by
$$
\langle fr,v\rangle=\langle f,rv\rangle,\quad r\in\mathcal{R},f\in V^\ast,v\in V.
$$

Assume that $\mathcal{R}$ has a degree-inverting anti-automorphism $\psi$. Then $V^\ast$ becomes a left $\mathcal{R}$-module by
\begin{equation}\label{leftVast}
r\cdot f:=f\psi(r),\quad r\in\mathcal{R},f\in V^\ast.
\end{equation}
\begin{Lemma}\label{lem2}
With \eqref{leftVast}, $V^\ast$ is an inverted-graded left $\mathcal{R}$-module, that is, $V^\ast$ is a left $\mathcal{R}$-module and
$$
\mathcal{R}_g\cdot V_t^\ast\subseteq V^\ast_{tg^{-1}},\quad\forall g,t\in G.
$$
\end{Lemma}
\begin{proof}
Let $r\in\mathcal{R}_g$, $f\in V^\ast_t$, $v\in V_h$. Then
$$
\deg(r\cdot f)h=\deg\langle r\cdot f,v\rangle=\deg\langle f\psi(r),v\rangle=\deg\langle f,\psi(r)v\rangle=tg^{-1}h,
$$
thus, $\mathcal{R}_gV_t^\ast\subseteq V_{tg^{-1}}^\ast$.
\end{proof}

For any $G$-graded vector space $W=\bigoplus_{g\in G}W_g$, we define $W^{[-]}=\bigoplus_{g\in G}W_g^{[-]}$, where $W_g^{[-]}=W_{g^{-1}}$. These are known as \emph{Veronese modules} (see \cite[Example 1.2.7]{Hazbook}, for a more general construction).
\begin{Lemma}\label{lem3}
$V$ is an inverted-graded left $\mathcal{R}$-module if and only if $V^{[-]}$ is a graded left $\mathcal{R}$-module.
\end{Lemma}
\begin{proof}
Assume that $V^{[-]}$ is a graded left $\mathcal{R}$-module. Then
$$
\mathcal{R}_gV_t=\mathcal{R}_gV^{[-]}_{t^{-1}}\subseteq V^{[-]}_{gt^{-1}}=V_{tg^{-1}}.
$$
Conversely, if $V$ is an inverted-graded left $\mathcal{R}$-module, then
$$
\mathcal{R}_gV^{[-]}_t=\mathcal{R}_gV_{t^{-1}}\subseteq V_{t^{-1}g^{-1}}=V^{[-]}_{gt}.
$$
\end{proof}
\begin{Lemma}\label{lem4}
There exists a degree-inverting $\mathcal{R}$-isomorphism $\varphi_1:V^{[g_0]}\to V^\ast$, for some $g_0\in G$. Equivalently, $\varphi_1:V^{[g_0]}\to V^{\ast[-]}$ is a $G$-graded $\mathcal{R}$-isomorphism.
\end{Lemma}
\begin{proof}
It follows from Lemma \ref{lem3} and Lemma 2.7 of \cite{EK2013}.
\end{proof}
From now on, we fix $g_0\in G$ and $\varphi_1:V^{[g_0]}\to V^\ast$, as in Lemma \ref{lem4}.

\begin{Lemma}\label{lem5}
There exists a homogeneous anti-automorphism $\psi_0:\mathcal{D}\to\mathcal{D}$ such that 
\begin{equation}\label{eqlem5}
\varphi_1(vd)=\psi_0(d)\varphi_1(v),
\end{equation}
for all $v\in V$, $d\in\mathcal{D}$. Moreover, $\deg\psi_0(d)=g_0^{-1}(\deg d)^{-1}g_0$, for any nonzero homogeneous $d\in D$.
\end{Lemma}
\begin{proof}
For any homogeneous $d\in\mathcal{D}$, let $R_d:V\to V$ be the right multiplication by $d$, and $L_d:V^\ast\to V^\ast$ the left multiplication. We will prove that the following sets coincide:
\begin{align*}
&S_1=\{\varphi:V^{[g]}\to V^\ast\text{ degree-inverting $\mathcal{R}$-isomorphism, for some $g\in G$}\},\\%
&S_2=\{\varphi_1\circ R_d\mid\text{ $d\in\mathcal{D}^\times$ homogeneous}\},\\%
&S_3=\{L_d\circ\varphi_1\mid\text{ $d\in\mathcal{D}^\times$ homogeneous}\}.
\end{align*}
It is clear that $S_2,S_3\subseteq S_1$. Given $\varphi\in S_1$, we have $\varphi_1^{-1}\circ\varphi\in\mathrm{End}_\mathcal{R}(V)\cong\mathcal{D}$. Thus, for some nonzero homogeneous $d\in\mathcal{D}$, we have $\varphi_1^{-1}\circ\varphi=R_d$; which implies $\varphi=\varphi_1\circ R_d\in S_2$. Similarly, $\varphi\circ\varphi_1^{-1}\in\mathrm{End}_\mathcal{R}(V^\ast)\cong\mathcal{D}$, so we can find a nonzero homogeneous $d\in\mathcal{D}$ such that $\varphi\circ\varphi_1^{-1}=L_d$. Hence, $\varphi=L_d\circ\varphi_1\in S_3$.

Now, since $S_2=S_3$, given a nonzero homogeneous $d\in\mathcal{D}$, we can find a nonzero homogeneous $d'\in\mathcal{D}$ such that $L_{d'}\circ\varphi_1=\varphi_1\circ R_d$. Define $\psi_0:\mathcal{D}\to\mathcal{D}$ linearly, such that $\psi(d)=d'$. By construction, $\psi_0$ is a linear isomorphism, and it is an anti-homomorphism. Also, $L_{\psi(d)}\circ\varphi_1=\varphi_1\circ R_d$ is equivalent to $\psi(d)\varphi_1(v)=\varphi_1(vd)$, for all $v\in V$. Moreover, from this relation, we derive the following:
$$
\deg\psi(d)\left((\deg v)g_0\right)^{-1}=\left((\deg v)(\deg d)g_0\right)^{-1}.
$$
Or, equivalently, $\deg\psi(d)=g_0^{-1}(\deg d)^{-1}g_0$.
\end{proof}
\begin{Remark}\label{remarkafterlemma}
If it happens that $g_0\in\mathrm{Supp}\,\mathcal{D}$, then, by the proof of Lemma \ref{lem5}, we can replace $\varphi_1$ by $\varphi_1\circ R_{d_0}$, where $d_0\in\mathcal{D}$ is homogeneous with $\deg d_0=g_0$. Thus, $\deg\psi_0(d)=(\deg d)^{-1}$ for all homogeneous $d$, so that the new $\psi_0:\mathcal{D}\to\mathcal{D}$ is a degree-inverting involution on $\mathcal{D}$.
\end{Remark}


Now, we have a non-degenerate $\mathbb{F}$-bilinear form $B:V\times V\to\mathcal{D}$ given by
$$
B(v,w)=\langle\varphi_1(v),w\rangle.
$$
This form satisfies the following properties:
\begin{enumerate}
\renewcommand{\labelenumi}{(\roman{enumi})}
\item $\deg B(v,w)=g_0^{-1}(\deg v)^{-1}\deg w$, for all homogeneous $v,w\in V$,
\item $B$ is $\psi_0$-sesquilinear, that is, $B(vd,w)=\psi_0(d)B(v,w)$, $B(v,wd)=B(v,w)d$, $v,w\in V$, $d\in\mathcal{D}$,
\item $B(rv,w)=B(v,\psi(r)w)$, $v,w\in V$, $r\in\mathcal{R}$.
\end{enumerate}

Conversely, a pair $(B,\psi_0)$ satisfying (i)--(iii) determines uniquely $\psi$, that is, we can recover $\psi$ from the pair $(B,\psi_0)$. Indeed, let $\{w_1,\ldots,w_n\}$ be a homogeneous $\mathcal{D}$-basis of $V$. Let $\Phi=(x_{ij})$, where $x_{ij}=B(w_i,w_j)$, be the matrix of $B$. Given $r\in\mathcal{R}$, let $R=(r_{ij})$ be its matrix form, and $\psi(R)=(r_{ij}')$ the matrix form of $\psi(r)$. Then, we have
\begin{align*}
&B(rw_k,w_\ell)=B(\sum_{i=1}^nw_ir_{ik},w_\ell)=\sum_{i=1}^n\psi_0(r_{ik})x_{i\ell}\\
&B(w_k,\psi(r)w_\ell)=B(w_k,\sum_{i=1}^nw_ir_{i\ell}')=\sum_{i=1}^nx_{ki}r_{i\ell}'
\end{align*}
So, we obtain the equation $\psi_0(R)^t\Phi=\Phi R$. Hence,
\begin{equation}\label{PSI}
\psi:X\in\mathcal{R}\mapsto\Phi^{-1}\psi_0(X^t)\Phi\in\mathcal{R},
\end{equation}
where we identify, via Kronecker product, $\mathcal{R}=M_n(\mathcal{D})$, $\psi_0(X)$ means that we are applying $\psi_0$ in the entries of $X$, and $^t$ is the usual matrix transposition of the $n\times n$ matrices $M_n(\mathcal{D})$. 

We summarize the results obtained so far:
\begin{Prop}[{cf. \cite[Theorem 2.57]{EK2013}}]
Let $G$ be any group, $\mathcal{D}$ a graded division algebra, $V$ a finite-dimensional graded right $\mathcal{D}$-module and $\mathcal{R}=\mathrm{End}_\mathcal{D}(V)$. Assume that $\psi$ is a degree-inverting anti-automorphism of $\mathcal{R}$. Then there exist $g_0\in G$, an anti-automorphism $\psi_0$ on $\mathcal{D}$ satisfying $\deg\psi_0(d)=g_0^{-1}(\deg d)^{-1}g_0$ for all homogeneous $d\in\mathcal{D}$, and a non-degenerate form $B:V\times V\to\mathcal{D}$ satisfying (i)--(iii). If $(\psi_0',B')$ is another such pair, then there exists a nonzero homogeneous $d\in\mathcal{D}$ such that $B'=dB$ and $\psi_0'(x)=d\psi_0(x)$, $\forall x\in\mathcal{D}$.

Conversely, given a pair $(\psi_0,B)$ satisfying (i)--(iii), there exists a degree-inverting anti-automorphism on $\mathcal{R}$.\qed
\end{Prop}

Now, from now on, we assume that $\psi$ is a degree-inverting \emph{involution}, that is, $\psi^2=1$.

\begin{Lemma}\label{lem7}
If $\psi$ is an involution, then
$$
B(w,v)=\varepsilon_B\psi_0(B(v,w)),\quad\forall v,w\in V,
$$
where $\varepsilon_B\in\{1,-1\}$.
\end{Lemma}
\begin{proof}
Define $\bar{B}(v,w)=\psi_0(B(w,v))$. Then $\bar B$ is a non-degenerate $\psi_0$-sesquilinear form satisfying (ii). Thus, we can find an invertible $\mathcal{D}$-linear $Q:V\to V$ such that $\bar{B}(v,w)=B(Qv,w)$, for all $v,w\in V$. Hence, for any $r\in\mathcal{R}$, $v,w\in V$,
\begin{align*}
B(v,rw)&=B(\psi(r)v,w)=\psi_0\bar{B}(w,\psi(r)v)=\psi_0B(Qw,\psi(r)v)=\psi_0B(rQw,v)=\\
&=\bar{B}(v,rQw)=B(Qv,rQw).
\end{align*}
Taking $r=1$, we see that $B(v,w)=B(Qv,Qw)$ for all $v,w\in V$. Hence, we have
$$
B(v,rw)=B(Qv,rQw)=B(v,Q^{-1}rQw).
$$
So $r=Q^{-1}rQ$, for all $r\in\mathcal{R}$. This gives $Q=\lambda\in\mathbb{F}$. Moreover, $B(v,w)=\lambda^2 B(v,w)$, for all $v,w\in V$, which implies $\lambda\in\{1,-1\}$. Thus, $\psi_0B(w,v)=\bar{B}(v,w)=\varepsilon_BB(v,w)$, where $\varepsilon_B=\lambda$.
\end{proof}
As a result, $B$ is \emph{balanced}, that is, $B(v,w)=0$ if and only if $B(w,v)=0$.

Given any $\mathcal{D}$-subspace $U\subseteq V$, we define
$$
U^\perp=\{x\in V\mid B(x,U)=0\}=\{x\in V\mid B(U,x)=0\}.
$$

The following result is standard:
\begin{Lemma}\label{lem8}
Let $B:V\times V\to\mathcal{D}$ be a non-degenerate balanced $\mathbb{F}$-bilinear form. Given a $\mathcal{D}$-subspace $U\subseteq V$, we have  $V=U\oplus U^\perp$ if and only if $B|_U$ is non-degenerate.\qed
\end{Lemma}

Now, using Lemma \ref{lem8}, we can construct a homogeneous $\mathcal{D}$-basis of $V$
\begin{equation}\label{basis}
\{v_1,\ldots,v_m,v_{m+1}',v_{m+1}'',\ldots,v_s',v_s''\},
\end{equation}
satisfying
\begin{enumerate}
\renewcommand{\labelenumi}{(\alph{enumi})}
\item $B(v_i,v_i)\ne0$, $i=1,2,\ldots,m$,
\item $B(v_j',v_j'')=1$, $j>m$,
\item all the remaining $B(v,w)=0$.
\end{enumerate}
Let $g_i=\deg v_i$, $g_j'=\deg v_j'$, $g_j''=\deg v_j''$. If $m>0$, then $T\ni\deg B(v_1,v_1)=g_0^{-1}$. Also,
$$
1=\deg B(v_j',v_j'')=g_0^{-1}g_j^{\prime-1}g_j'',
$$
so $g_j''=g_j'g_0$, for all $j>m$. Moreover, we have
\begin{Lemma}\label{lem11}
If $s>m$, then $g_0^2=1$.
\end{Lemma}
\begin{proof}
Since
$$
B(v_s'',v_s')=\varepsilon_B\psi_0(B(v_s',v_s''))=\varepsilon_B1,
$$
we obtain $1=g_0^{-1}(g'')^{-1}g'=g_0^{-2}$. Thus, $g_0^2=1$.
\end{proof}

Now, if $\varepsilon_B=1$ then we call $\psi$ orthogonal, and otherwise, $\psi$ is symplectic. We note that $\varepsilon_B=-1$ implies $m=0$ in the previous notations. Using \eqref{PSI}, we can construct the matrix of $\Phi$, and determine $\psi$ in matrix form. It will be convenient to use the basis $\{v_1,\ldots,v_m,v_{m+1}',\ldots,v_s',v_{m+1}'',\ldots,v_s''\}$. We summarize the results

\begin{Thm}
Let $\mathcal{R}=M_n(\mathcal{D})$ be a matrix algebra endowed with a $G$-grading parametrized by $(\mathcal{D},\gamma)$. Then $\mathcal{R}$ admits a degree-inverting involution $\psi$ if and only if there exists $g_0\in G$, the graded division algebra $\mathcal{D}$ admits an involution $\psi_0$ satisfying $\deg\psi_0(d)=g_0^{-1}(\deg d)^{-1}g_0$, $\forall d\in\mathcal{D}$ homogeneous, and
$$
\gamma=(g_1,\ldots,g_m,g_{m+1}',\ldots,g_s',g_{m+1}'',\ldots,g_s'')
$$
where $g_j''=g_j'g_0$, for all $j>m$. Moreover, if $g_0\notin T$ then $m=0$; if $g_0\in T$, then we can assume $\psi_0$ a degree-inverting involution; and if $s>m$ then $g_0^2=1$.

Let $\{X_u\mid u\in T\}$ be a homogeneous basis of $\mathcal{D}$. In any case, $\psi(e_{ij}\otimes X)=\Phi^{-1}e_{ji}\otimes\psi_0(X)\Phi$, for $e_{ij}\otimes X\in\mathcal{R}$, where $\Phi$ is given by:

(i) if $\psi$ is orthogonal,
$$
\Phi=\left(\begin{array}{ccc}I_m\otimes X_{g_0}\\&0&I_s\otimes X_1\\&I_s\otimes X_1&0\end{array}\right).
$$

(ii) if $\psi$ is symplectic, then
$$
\Phi=\left(\begin{array}{cc}0&I_s\\-I_s&0\end{array}\right)\otimes X_1.
$$\qed
\end{Thm}

\begin{Remark}
It is worth mentioning that, if $G$ is assumed to be abelian, then we obtain a complete description of degree-inverting involutions on $M_n(\mathcal{D})$: the involution $\psi_0$ on $\mathcal{D}$ will be degree-inverting, and we apply Corollary \ref{commutativegradeddiv}.
\end{Remark}

\section{Degree-inverting involution on upper triangular matrices\label{utn}}
In this section we shall classify degree-inverting involution on the algebra of upper triangular matrices. The final result is similar to the degree-preserving involution case \cite{DVKS}. However, in the degree-inverting case, the support of the grading does not need to be commutative. We shall improve the result obtained in \cite{FM2017}, since we only impose the restriction $\mathrm{char}\,\mathbb{F}\ne2$.

Let $\mathbb{F}$ be an arbitrary field of characteristic not $2$, and $G$ any group. It is known that every group grading on $UT_n$ is elementary \cite{VZ2007}, that is, every grading admits an isomorphic structure where each matrix unit $e_{ij}$ is homogeneous. Moreover, an isomorphism class of $G$-gradings on $UT_n$ is uniquely determined by a sequence $\eta=(g_1,\ldots,g_{n-1})\in G^{n-1}$, where $\deg e_{i,i+1}=g_i$, for $i=1,2,\ldots,n-1$ (see \cite[Theorem 2.3]{DVKV2004}).

From now on, we fix a $G$-grading on $UT_n$, given by $\eta=(g_1,g_2,\ldots,g_{n-1})$. Let $J=J(UT_n)$ be the Jacobson radical, which is clearly a graded ideal. We denote by $\tau$ the canonical involution of $UT_n$, that is, $\tau(e_{ij})=e_{n-j+1,n-i+1}$. Note that $\tau$ is the flip along the secondary diagonal of $M_n$.

Let $\rho$ be a degree-inverting involution of $UT_n$. Since $\rho(J^m)=J^m$, for every $m\ge1$, we have that $\rho$ is a degree-inverting involution on $J/J^2$. Moreover, we know that every automorphism of $UT_n$ is inner (see, for instance, \cite{Jondrup}); hence, $\rho=\mathrm{Int}(u)\circ\tau$, for some inner automorphism $\mathrm{Int}(u)$ (where $u\in UT_n$ is invertible). Thus, $\rho(e_{i,i+1}+J^2)=e_{n-i,n-i+1}+J^2$; that is, $\deg e_{i,i+1}=\left(\deg e_{n-i,n-i+1}\right)^{-1}$. This proves

\begin{Lemma}\label{lem1}
$(UT_n,\eta)$ admits a degree-inverting involution if, and only if, $g_i=g_{n-i+1}^{-1}$ for each $i=1,2,\ldots,\lceil\frac{n}2\rceil$.
\end{Lemma}
\begin{proof}
The argument above proves the ``only if" part. The ``if" part is obvious, since $\tau$ will invert degree, under this condition.
\end{proof}
\begin{Remark}
Note that, in contrast with the graded-involution case, the existence of a degree-inverting involution does not imply that the support of the grading is commutative.
\end{Remark}

Now, assume from now on that $\eta$ satisfies the condition of Lemma \ref{lem1}. It is clear that $\tau$ is a degree-inverting involution in this case. Since we wrote $\rho=\mathrm{Int}(u)\circ\tau$, we note that $\mathrm{Int}(u)$ is a graded automorphism of $UT_n$. Thus, $u$ is homogeneous of degree $1$. Moreover, since $\rho^2=1$, one has $\tau(u)=\pm u$. We note that $\tau(u)=-u$ happens only if $n$ is even. Indeed, if $n=2m+1$, then $\tau(e_{m+1,m+1})=e_{m+1,m+1}$. Since $u$ is invertible, the entry $(m+1,m+1)$ of $u$ must be nonzero; and at the same time, it should coincide with its opposite, a contradiction.

Suppose $n=2m$, and let $D=\mathrm{diag}(1,\ldots,1,-1,\ldots,-1)$. The involution $s(x)=D\tau(x)D$ is called the \emph{symplectic involution} of $UT_n$.

Finally, if $n=2m+1$, then we can multiply $u$ by some scalar (note that, $\mathrm{Int}(u)=\mathrm{Int}(\lambda u)$), in such a way that its $m+1$ entry is $1$ (this is an important step in the proof of the next lemma, see \cite[Lemma 2.4]{DVKS}). Also, if $\tau(u)=-u$, then
$$
\rho(x)=u\tau(x)u^{-1}=uDD\tau(x)DDu^{-1}=\mathrm{Int}(uD)(s(x)).
$$
In this case, $s(uD)=uD$. So, we can replace $u$ by $uD$ to obtain $s(u)=u$. Hence, in any case, we always obtain the equation
$$
\rho=\mathrm{Int}(u)\circ\rho_0,
$$
with $\rho_0(u)=u$, where $\rho_0$ is either $\tau$ or $s$.

\begin{Lemma}
Assume $\mathrm{char}\,\mathbb{F}\ne2$. Let $u$ be an invertible homogeneous element of degree $1$. Let $\rho_0$ be either $\tau$ or $s$, in such a way that $\rho_0(u)=u$; and if $n=2m+1$, assume that the entry $(m+1,m+1)$ of $u$ is $1$. Then there exists a homogeneous invertible element $v\in UT_n$, of degree 1, such that $u=v\rho_0(v)$.
\end{Lemma}
\begin{proof}
The proof is exactly the construction of the proof of Lemma 2.4 of \cite{DVKS} (see also \cite[Lemma 6.9]{FM2017}). As an example, we include here the case $n=2m$, and $\rho_0=\tau$. Write
$$
u=\left(\begin{array}{cc}
X&Z\\
0&Y
\end{array}\right),
$$
where $X,Y\in UT_m$ are invertible, and $Z\in M_m$. Then
$$
v=\left(\begin{array}{cc}
\mathrm{Id}_m&\frac12Z\\
0&Y
\end{array}\right)
$$
satisfies $u=v\tau(v)$. Moreover, let $\mathcal{U}$ be the set of pairs $(i,j)$ such that $u=\sum_{(i,j)\in\mathcal{U}}\alpha_{ij}e_{ij}$, for $\alpha_{ij}\ne0$. Since $u$ is homogeneous of degree $1$, and every matrix unit is homogeneous; $\deg e_{ij}=1$, for all $(i,j)\in\mathcal{U}$. Now, by construction, $v=\sum_{(i,j)\in\mathcal{U}'}\beta_{ij}e_{ij}$, for some $\mathcal{U}'\subseteq\mathcal{U}$. In particular, $v$ is a linear combination of homogeneous elements of degree $1$. This imply $v$ homogeneous of degree $1$.

The proof is similar for the other cases.
\end{proof}

As a conclusion, $\rho=\mathrm{Int}(u)\circ\rho_0=\mathrm{Int}(v)\circ\mathrm{Int}(\rho_0(v))\circ\rho_0$, where $\rho_0$ is either $\tau$ or $s$, and $\rho_0(u)=u$. A straightforward argument shows that, in this case, $\rho$ is equivalent to $\rho_0$. Indeed, we need to find a graded automorphism $\varphi$ such that $\varphi(\rho_0(x))=\rho(\varphi(x))$. Taking $\varphi=\mathrm{Int}(v)$, we have
$$
\rho(\varphi(x))=\rho(vxv^{-1})=\mathrm{Int}(v)\mathrm{Int}(\rho_0(v))\rho_0(vxv^{-1})=\mathrm{Int}(v)(\rho_0(x))=\varphi(\rho_0(x)).
$$
We summarize our main result of this section:
\begin{Thm}
Let $\mathbb{F}$ be a field of characteristic not $2$, and $G$ any group. Let $(UT_n,\eta)$ be $G$-graded, where $\eta=(g_1,g_2,\ldots,g_{n-1})$. Then $(UT_n,\eta)$ admits a degree-inverting involution if, and only if, $g_i=g_{n-i+1}^{-1}$, for all $i=1,2,\ldots,n-1$. In this case, every degree-inverting involution is equivalent either to $\tau$ or to $s$; where $s$ can occur if, and only if, $n$ is even.\qed
\end{Thm}
Our definition of elementary grading on $UT_n$ is not the standard one. Usually one defines an elementary grading on $UT_n$ as we did for matrix algebras, that is, a sequence $\gamma=(h_1,h_2,\ldots,h_n)\in G^n$ defines a $G$-grading on $UT_n$ by $\deg e_{ij}=h_ih_j^{-1}$. However, we cannot find a friendly way to write the condition of existence of a degree-inverting involution on $UT_n$ in the standard notation. Nonetheless, if the grading group is abelian then the condition is nicely written, and we reobtain a result of \cite{FM2017}:
\begin{Cor}[{\cite[Corollary 5.11]{FM2017}}]
Let $\mathbb{F}$ be a field of characteristic not $2$, and $G$ be an abelian group. Let $UT_n$ be endowed with an elementary $G$-grading given by $\gamma=(h_1,\ldots,h_n)$. Then $UT_n$ admits a degree-inverting involution if and only if $h_1h_n^{-1}=h_2h_{n-1}^{-1}=\cdots=h_nh_1^{-1}$. In this case, every degree-inverting involution is equivalent either to $\tau$ or to $s$; where $s$ can occur if, and only if, $n$ is even.\qed
\end{Cor}

\end{document}